\documentclass{amsart}

\usepackage{graphicx}
\usepackage{amsfonts}
\usepackage{amsthm}
\usepackage{geometry}
\usepackage{amsmath}
\usepackage{txfonts}
\usepackage{wrapfig}
\usepackage{ascmac}
\usepackage{bm}
\usepackage{indentfirst}
\usepackage{tikz}
\usepackage{color}
\usepackage{ulem}
\usepackage{comment}

\newtheorem{thm}{Theorem}[section]
\newtheorem{rem}[thm]{Remark}
\newtheorem{defi}[thm]{Definition}
\newtheorem{ex}[thm]{Example}
\newtheorem{lem}[thm]{Lemma}
\newtheorem{prop}[thm]{Proposition}

\newtheorem{cor}[thm]{Corollary}

\newtheorem*{thmA*}{Theorem A}
\newtheorem*{thmB*}{Theorem B}

\def\A{{\mathcal A}}
\def\B{{\mathcal B}}
\def\L{{\mathcal L}}

\newcommand{\abs}[1]{\left\lvert#1\right\rvert} 
\newcommand{\rk}{\mathrm{rank}} 
\newcommand{\rank}{\mathrm{rank}} 

\title{Arithmetic non-very generic arrangements}
\author{Pragnya Das\and Takuya Saito\and Simona Settepanella}

\address[Das]{Department of Mathematics, HBS, Indian Institute of Information Technology, Manipur, India}
\address[Saito]{Institute for Chemical Reaction Design and Discovery, Hokkaido University}
\address[Settepanella]{Department of Statistics and Economics, Torino University, Italy}

\email[Das]{pragnya@iiitmanipur.ac.in}
\email[Saito]{saito@icredd.hokudai.ac.jp}
\email[Settepanella]{simona.settepanella@unito.it}


\date{\today}

\begin{document}
\maketitle
\begin{abstract}
A discriminantal hyperplane arrangement $\mathcal{B}(n,k,\mathcal{A})$ is constructed from a given (generic) hyperplane arrangement $\mathcal{A}$, which is classified as either very generic or non-very generic depending on the combinatorial structure of $\mathcal{B}(n,k,\mathcal{A})$. In particular, $\mathcal{A}$ is considered non-very generic if the intersection lattice of $\mathcal{B}(n,k,\mathcal{A})$ contains at least one non-very generic intersection - that is, an intersection that fails to satisfy a specific rank condition established by Athanasiadis in \cite{Atha}. In this paper, we present arithmetic criteria characterizing non-very generic intersections in discriminantal arrangements and we complete and correct a previous result by Libgober and the third author concerning rank-two intersections in such arrangements.\end{abstract}

\section{Introduction}

The discriminantal arrangement $\mathcal{B}(n,k,\mathcal{A})$, introduced by Manin and Schechtman in 1989, is an arrangement of hyperplanes $D_L$ that consists of non-generic parallel translates of a generic arrangement $\mathcal{A}$ of $n$ hyperplanes in a $k$-dimensional space (see \cite{MS}). Their original construction focused on cases where the combinatorics of $\mathcal{B}(n,k,\mathcal{A})$ remain constant as $\mathcal{A}$ varies within an open Zariski subset $\mathcal{Z}$ of all generic arrangements. However, cases where the combinatorics change were left unexamined. Falk was the first to highlight this distinction (see \cite{Falk}), and Bayer and Brandt further clarified it in \cite{BB}, introducing the terms \emph{very generic} for arrangements $\mathcal{A} \in \mathcal{Z}$ and \emph{non-very generic} otherwise.
\\
In \cite{Atha}, Athanasiadis proved a conjecture by Bayer and Brandt, giving a complete description of the combinatorics of $\mathcal{B}(n,k,\mathcal{A})$ when $\mathcal{A}$ is very generic. This condition which we call the \emph{Bayer-Brandt-Athanasiadis condition} (or BBA-condition), characterizes the intersection lattice $\mathcal{L}(\mathcal{B}(n,k,\mathcal{A}))$ in the very generic case. \\
One early example of a non-very generic arrangement (see \cite{crapo}) is the Crapo's configuration of six lines in $\mathbb{R}^2$ depicted in Figure~\ref{fig:crapo}, in which the intersection lattice $\mathcal{L}(\mathcal{B}(6,2,\mathcal{A}))$ contains a rank 3 intersection of multiplicity 4 which fails the BBA-condition. 
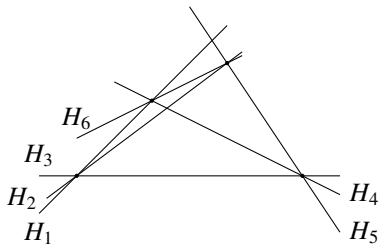
\begin{figure}[h!]
\centering
\begin{tikzpicture}[scale=0.50]
\coordinate (A) at (0, 0);
\coordinate (B) at (6, 0);
\coordinate (C) at (4, 3);
\coordinate (D) at (2, 2);
\draw (-1,-1) node[below] {$H_1$} -- (A)--(D) -- (4,4);        
\draw (-0.8,-0.6) node[left] {$H_2$} -- (4.4,3.3);            
\draw (-1,0) node[above] {$H_3$} -- (7,0);                    
\draw (1,2.5) -- (7,-0.5) node[right] {$H_4$};                
\draw (7,-1.5) node[right] {$H_5$} -- (3,4.5);                
\draw (0,1) node[above] {$H_6$} -- (D) -- (4.4,3.2);           
\foreach \point in {A,B,C,D}
\fill [black] (\point) circle (1.6pt);
\end{tikzpicture}
\caption{Crapo's arrangement of six lines}
\label{fig:crapo}
\end{figure}
\\
However, the BBA-condition is global and does not allow one to determine locally whether a given intersection is very generic. Indeed Yamagata and the third author recently gave two examples of non-very generic arrangements (see Examples 4.2 and 4.3 in \cite{SeSo}) where certain intersections in $\mathcal{L}(\mathcal{B}(n,k,\mathcal{A}))$ satisfy the BBA-condition, but are non-very generic. 
These results demonstrate that the failure of the BBA-condition is sufficient - but not necessary - for an intersection to be non-very generic. Motivated by this, we introduce the following distinction:

\bigskip
\textit{A non-very generic intersection is called \emph{arithmetic non-very generic} (ANVG) if its combinatorics fail the BBA-condition and \emph{geometric non-very generic} (GNVG) if its combinatorics satisfy the BBA-condition.}
\bigskip

\noindent
Consequently we define a non-very generic arrangement \emph{arithmetic non-very generic} (ANVG) if all its non-very generic intersections are ANVG and \emph{geometric non-very generic} (GNVG) otherwise. For example, Crapo's arrangement in Figure~\ref{fig:crapo} is ANVG, while the examples in \cite{SeSo} are GNVG arrangements.
\\
In this paper we focus on  ANVG arrangements. In particular in Lemma~\ref{nonvg-bba}, we characterize when the failure of the BBA-condition is both necessary and sufficient for an intersection to be non-very generic.
\\
More in details, in Section~\ref{sec:preli}, we review the definitions and properties of discriminantal arrangements. Section~\ref{sec:NVI} discusses the BBA-condition and criteria for ANVG intersections. In Section~\ref{sec:MNG}, we define and analyze  \emph{minimal non-very generic intersections}, which satisfy structural properties that lead to the inequality criteria summarized in Theorem~\ref{Thm:B}. These inequalities are particularly useful for constructing new families of non-very-generic arrangements. Section~\ref{sec:rank2} presents a complete classification of rank-2 intersections in discriminantal arrangements. In particular, by means of Falk's adjoint construction (see \cite{Falk}), we construct arrangements - starting with 8 hyperplanes in $\mathbb{R}^5$ - in which rank-2 intersections of multiplicity 4 arise. This provides counterexamples that correct a claim by Libgober and Settepanella \cite{LS}, who stated that rank-2 intersections in $\mathcal{B}(n,k,\mathcal{A})$ can only have multiplicity $2$, $3$, or $k+2$.

\section*{Acknowledgment}
This work is supported by JSPS KAKENHI Grant Number JP23KJ0031.

\section{Preliminaries}\label{sec:preli}
Throughout this paper, we choose $\mathbb{R}$ to be the coefficient field for convinience, but the same discussion can be made if we assume other fields instead of $\mathbb{R}$. \\
Let $\A = \{ H_1, \dots, H_n \}$ be a central arrangement in $\mathbb{R}^k,k<n$, such that any $m$ hyperplanes intersect in codimension $m$ at any point except for the origin for any $m \leq k$. We will call such an arrangement a generic arrangement.
In the rest of this paper, we assume $\A$ to be a generic arrangement consisting of hyperplanes indexed by $[n]=\{1,\ldots,n\}$ in $\mathbb{R}^k$ unless otherwise mentioned. Given any arrangement $\mathcal{A}$ we denote by $\mathcal{L}(\mathcal{A})$ its intersection poset.
\subsection{Discriminantal Arrangements and its combinatorics.} Let $\A=\{H^0_i\}_{i \in [n]}$ be a generic arrangement of $n$ hyperplanes in $\mathbb{R}^k, k<n$ and $\alpha_i \in \mathbb{R}^k$ the normal vectors to the hyperplanes $H^0_i, i \in [n]$. 
The space $\mathbb{S}(H_1^0,...,H_n^0)$ (or simply $\mathbb{S}$ when the 
dependence on $H_i^0$ is clear or not essential) is the space of parallel translation $\A^t$ of the arrangement $\A$. In other words, $\mathbb{S}$ is the space of $n$-tuples $(H_1^{t_1},\ldots,H_n^{t_n})$, where $H_i^{t_i}=H^0_i+t_i\alpha_i$ is a parallel translation of $H_i^0\in\A$ by $t_i \in\mathbb{R}$. For any subset of indices $L \subset [n]$, with $|L|=k+1$, the set $D_L=\{(H_1^{t_1},\ldots,H_n^{t_n})\in \mathbb{S} \mid \cap_{i\in L} H^{t_i}_i\neq \phi\}$ is a hyperplane in the space $\mathbb{S}$ (see  \cite{MS}). The discriminantal arrangement $\B(n, k, \A)$ is defined as: 
$$\B(n, k, \A)=\{D_L \subset  \mathbb{S} \mid L \subset [n], | L |=k+1\} \quad .$$
Notice that there is a natural identification of $\mathbb{S}$ with the $n$-dimensional affine space $\mathbb{R}^n$ given by the correspondence of $(H_1^{t_1},\ldots,H_n^{t_n})$ with the $n$-tuple $(t_1,\ldots, t_n)$.\\
It is well known (see, among others \cite{MS},\cite{crapo}) that there exists an open Zariski set $\mathcal{Z}$ in the space of generic arrangements of $n$ hyperplanes in $\mathbb{R}^k$, such that the $\L(\mathcal{B}(n,k,\A))$ is independent from the choice of the arrangement $\A \in  \mathcal{Z}$. Accordingly to Bayer and Brandt (see \cite{BB}) we will call the arrangements $\A \in  \mathcal{Z}$ \textit{very generic} and the others \textit{non-very generic}.

\subsection{Bayer-Brandt-Athanasiadis condition.}\label{BBA-c}In \cite{Atha}, Athanasiadis proved a conjecture given by Bayer and Brandt which stated that the intersection lattice of the discriminantal arrangement $\mathcal{B}(n,k,\A)$, when the arrangement $\mathcal{A}$ is very generic, is isomorphic to the collection of all sets $\{S_1, \ldots, S_m\}$, $S_i$ $\subset$ $[n]=\{1,\ldots,n\}$, $\abs{S_i} \geq k+1$ such that
\begin{equation}\label{eq:vgcon}
\abs{\bigcup_{i \in I} S_i} > k + \sum_{i \in I}(\abs{S_i}- k) \mbox{ for all } I \subset [m]=\{1,\ldots,m\}, \mid I \mid \geq 2 \quad .
\end{equation}
 The isomorphism is the natural one which associate to the set $S\subset[n]$ the space $D_S=\bigcap_{L \subset S, \mid L \mid=k+1} D_L, D_L \in \mathcal{B}(n,k,\A)$ of all translations $\mathcal{A}^t$  of $\A$ having hyperplanes indexed in $S$ intersecting in a not empty space. In particular $\{S_1, \ldots, S_m\}$ will correspond to the intersection $\bigcap_{i=1}^m D_{S_i}$. \\We call the condition provided in Equation \eqref{eq:vgcon} as the Bayer-Brandt-Athanasiadis-condition or the \textit{BBA-} condition. \\
Following Bayer, Brandt and Athanasiadis (see \cite{Atha},\cite{BB}), we denote by $P(n,k)$ the poset of all the sets $\{S_1, \ldots, S_m\}$, $S_i$ $\subset$ $[n]=\{1,\ldots,n\}$, $\abs{S_i} \geq k+1$ satisfying the Equation \eqref{eq:vgcon}. The elements in $P(n,k)$ are partially ordered by $S\preceq T$ if every subet in $S$ is contained in some subet of $T$. \\
For a generic arrangement $\A$ and an index subset $S\subset[n], \abs{S}>k$, it holds that the $\rank D_S=\abs{S}-k$. Hence if $\mathcal{A}$ is very generic and the BBA-condition is satisfied then the subspaces $D_{S_i},i\in\{1,\ldots,m\}$ intersect transversally (see \cite{Atha}) or, equivalently, 
\begin{equation}\label{eq:Atharank}
\rank\bigcap_{i=1}^m D_{S_i}=\sum_1^m (|S_i|-k).
\end{equation}

\section{Athanasiadis rank for non-very generic intersections}\label{sec:NVI}
An arrangement $\A$ is either very generic or non-very generic depending on the intersection lattice of $\B(n,k,\A)$. Hence, from this section onwards we will focus on the study of intersections $X\in\L(\B(n,k,\A)).$
\subsection{Canonical presentation of intersections in $\L(\B(n,k,\A))$}

The following definition will play a central role in the rest of the paper. 

\begin{defi}\label{X-max} 
Given an intersection $X\in\L(\B(n,k,\A))$ we call $D_S$ the \textbf{components} of $X$ if $X\subset D_S$ and $X\not\subset D_{S\cup\{j\}}$ for all $j\in[n]\setminus S$.
In particular for an intersection $X=\bigcap_{i=1}^m D_{S_i}\in\L(\B(n,k,\A))$, the set of indices $\{S_1,\ldots,S_m\}$ corresponding to all the components $\{D_{S_i},\ldots,D_{S_m}\}$  of $X$ is called the \textbf{canonical presentation} of $X$.
\end{defi}

\noindent
Let's remark that while there are several ways to represent an element $X \in \L(\B(n,k,\A))$, its components are unique and so is its canonical presentation. This motivates the following definition.

\begin{defi}[Athanasiadis rank]
Given an intersection $X=\bigcap_{i=1}^mD_{S_i}\in\L(\B(n,k,\A))$ with canonical presentation $\{S_1,\ldots,S_m\}$, we call the natural number $a_X=\sum_{i=1}^m(|S_i|-k)$ the \textbf{Athanasiadis rank} of $X$.
\end{defi}
\noindent
By Equation \eqref{eq:Atharank} if an arrangment $\A$ is very generic then all the intersections $X\in\L(\B(n,k,\A))$ satisfy $\rk(X)=a_X$.
\begin{defi}[Non-very generic intersection] 
Given an intersection $X=\bigcap_{i=1}^mD_{S_i}\in\L(\B(n,k,\A))$ with canonical presentation $\{S_1,\ldots,S_m\}$, we call $X$ very generic if $\rk(X)=a_X$ and non-very generic otherwise.
\end{defi}
\noindent

\noindent
While the first known examples of non-very generic intersections all failed the BBA-condition defined in Subsection \ref{BBA-c},  as shown in \cite{SeSo}, there are non-very generic intersections which satisfy it. Hence we further distinguish between the non-very generic intersections as follows.
\begin{defi}
We call a non-very generic intersection $X=\bigcap_{i=1}^mD_{S_i}\in\L(\B(n,k,\A))$ \textit{geometric non-very generic} or \textit{GNVG} if its canonical presentation $\{S_1,\ldots,S_m\}$ satisfies the BBA-condition and  \textit{arithmetic non-very generic} or \textit{ANVG} otherwise.
\end{defi}
\noindent
It is important to remark here that in \cite{Saito}, Lemma 3.8, second author shows that if $\A$ is a non-very generic arrangement then for any intersection $X \in \L(\B(n,k,\A))$ which is GNVG, it exists an intersection $Y \subset X$ which is ANVG. It follows that we can have non-very generic arrangements $\A$ which are \textit{arithmetic non-very generic}, i.e. all non-very generic intersections in $\L(\B(n,k,\A))$ are ANVG, but there are not \textit{geometric non-very generic} arrangements.\\
While the study of GNVG intersections seem to be very difficult, the one of the ANVG is simply related to the failing of the BBA-condition, i.e., to be an arithmetic non-very generic arrangement is completely classified by the BBA-condition. Hence our goal is to study under which conditions a non-very generic arrangement $\A$ is arithmetic. In the rest of this section we will focus on the study of conditions for which a non-very generic intersection is ANVG. 
\noindent
In order to simplify the exposition, from now on whenever we write $X=\bigcap_{i=1}^m D_{S_i}\in\L(\B(n,k,\A))$ we intend that $\{S_1,\ldots,S_m\}$ is the canonical presentation of $X$, unless differently specified. 

\begin{lem}\label{nonvg-bba}
Let $X=\bigcap_{i=1}^m D_{S_i}$ be an intersection in $\L(\B(n,k,\A))$ and $\mathbb{T}\in P(n,k)$ be a set which satisfies the condition: 
\begin{equation}\label{eq:ANVG}
   \begin{aligned}
\forall \, S_i \in \{S_1,\ldots,S_m\} \, \exists \, T\in\mathbb{T}\text{ such that }S_i\subset T,  \  \text{ and }  \  \rank(X)\geq\sum_{T\in\mathbb{T}} (\abs{T}-k)-\abs{\mathbb{T}_0},
   \end{aligned}
\end{equation} 
where $\mathbb{T}_0:=\{T\in\mathbb{T}\mid \abs{\{i\mid S_i\subset T\}}\geq 2\}$. Then $X$ is a non-very generic intersection if and only if it is an ANVG intersection.
\end{lem}
\begin{proof}
If canonical presentation $\{S_1,\ldots,S_m\}$ of the intersection $X=\bigcap_{i=1}^m D_{S_i}\in\L(\B(n,k,\A))$ doesn't satisfy the BBA-condition then $X$ is a non-very generic intersection. \\
Let $X=\bigcap_{i=1}^m D_{S_i}\in\L(\B(n,k,\A))$ be a non-very generic intersection.
     We define:
     \begin{equation*}
         X|T:=\bigcap_{S_i\subset T} D_{S_i}\mbox{ and } \epsilon_T:=\left\{\begin{aligned}0&&\mbox{ if }&T\notin\mathbb{T}_0\\ 1&&\mbox{ if }&T\in\mathbb{T}_0 \end{aligned}\right.
     \end{equation*}
    for $T\in\mathbb{T}$.
     Let us assume that \textit{BBA-} condition is satisfied then
     \begin{equation*}
         a_{X|T}\leq \abs{\bigcup_{S_i\subset T} S_i}-k-\epsilon_T\leq\abs{T}-k-\epsilon_T.
     \end{equation*}
     Thus $a_X=\sum_{T\in \mathbb{T}}a_{X|T}\leq \sum_{T\in\mathbb{T}}(\abs{\bigcup_{S_i\subset T} S_i}-k-\epsilon_T)\leq \sum_{T\in\mathbb{T}}(\abs{T}-k-\epsilon_T)=\sum_{T\in\mathbb{T}} (\abs{T}-k)-\abs{\mathbb{T}_0}\leq \rank(X)$ which is a contradiction. Hence the \textit{BBA-}condition fails.
\end{proof}
\noindent
A trivial example of this lemma are intersections $X=\bigcap_{i=1}^m D_{S_i}$ with $\rank(X)=3$ in an arrangement of $6$ lines in the real plane. In this case, if we fix $\mathbb{T}=\mathbb{T}_0=\{123456\}$, then all intersections $X=\bigcap_{i=1}^m D_{S_i}$ trivially satisfy the condition that $S_i\in\mathbb{T}$ and if $\rank(X)=3$ then the condition $\rank(X)\geq \sum_{T\in \mathbb{T}}(\abs{T}-2)-\abs{\mathbb{T}_0}=(6-2)-1=3$ is satisfied. That is all the non-very generic intersections in the Crapo's arrangement of $6$ lines in $\mathbb{R}^2$  fail the BBA-condition as in the example described in the Introduction (see Figure \ref{fig:crapo}).

\begin{lem}
Let $X=\bigcap_{i=1}^m D_{S_i}$ be an intersection of rank $n-k-1$ in $\L(\B(n,k,\A))$. If $X$ is a non-very generic intersection then $\abs{\bigcup_{i=1}^m S_i}=|\bigcup_{X\subset D_L}L|=n$. 
\end{lem}

\begin{proof}
 Let $X=\bigcap_{i=1}^m D_{S_i}$ be a non-very generic intersection of rank $n-k-1$ in $\L(\B(n,k,\A))$ and assume that $\abs{\bigcup_{i=1}^m S_i}=|\bigcup_{X\subset D_L}L|<n$. Then it exists a subset $L \in \binom{[n]}{k+1}\text{ such that } L\not\subset\bigcup_{i=1}^mS_i \text{ and hence }X\not\subset D_L$. It follows that the $\rk(X\cap D_L)=\rk(X)+\rk(D_L)=n-k,$ that is $X\cap D_L$ is central. On the other hand $X$ is a non-very generic intersection so the $\rk(X \cap D_L)< a_X+1=a_{X\cap D_L}.$ Hence $X\cap D_L$ is a central and non-very generic intersection which is absurd. 
\end{proof}
\noindent
The above lemma and Lemma \ref{nonvg-bba} when $\abs{\mathbb{T}}=1$ give us the following theorem:
\begin{thm}
Let $X=\bigcap_{i=1}^m D_{S_i}$ be an intersection of rank $n-k-1$ in $\L(\B(n,k,\A))$ 
 then $X$ is a non-very generic intersection if and only if it is ANVG.
\end{thm}
\begin{cor} An arrangement $\A$ with $n+3$ hyperplanes in $\mathbb{C}^n$ is non-very generic if and only if it is an ANVG arrangement.
\end{cor}

\section{Minimal non-very generic intersections}\label{sec:MNG}

In this section we define minimal non-very generic intersections in $\L(\B(n,k,A))$ and we study their properties. We recall here that when we write $X=\bigcap_{i=1}^m D_{S_i} \in \L(\B(n,k,\A))$ we intend that $\{S_1,\ldots,S_m\}$ is the canonical presentation of $X$, unless differently specified.

\begin{defi}[Minimal non-very generic intersection]
A non-very generic intersection $X\in\L(\B(n,k,\A))$ is \textbf{minimal} if and only if any intersection $Y\supsetneq X$ in $\L(\B(n,k,\A))$ is very generic.\end{defi}
\begin{rem}
A non-very generic intersection 
is either minimal or it is intersection of very generic intersections with minimal ones.
\end{rem}

\noindent
For an intersection $X=\bigcap_{i=1}^mD_{S_i}\in \L(\B(n,k,\A))$, we define 
\begin{equation*}
    \ell_X=\mathrm{min}_{j\in \bigcup_{i}S_i}\abs{\{i\in [m]\mid j\in S_i\}}.
\end{equation*}

\noindent
The following theorem is the main result of this section.
\begin{thm}\label{Thm:B}
    Let $X$ be a minimal non-very generic intersection in $\L(\B(n,k,\A))$, then:
    \begin{enumerate}
    	\item $2\leq \ell_X\leq k$ \textup{(Lemma \ref{prop:Bound_ell} and Lemma \ref{prop:UpperBoundOfell}),}
        \item $a_X\geq \ell_X\abs{\bigcup_{i}S_i}/(k+1)$ \textup{(Theorem \ref{prop:LowerBoundOfAthNum}),}
        \item $\rank(X)<a_X< \rank(X)+\ell_X$ \textup{(Theorem \ref{prop:UpperBoundOfAthNumb}),}
    \end{enumerate}
\end{thm}

The above inequalities yield the following corollary.

\begin{cor}\label{prop:NVRank}
    Let $X$ be a minimal non-very generic intersection in $\L(\B(n,k,\A))$, then 
    \begin{equation*}
        \rank(X)> \ell_X \left(\frac{\abs{\bigcup_{i}S_i}}{k+1}-1\right)\geq \frac{2\abs{\bigcup_{i}S_i}}{k+1}-2.
    \end{equation*}
\end{cor}
\begin{proof}
    By the theorems, it follows that
    \begin{align*}
        \rank(X)\geq a_X-\ell_X+1\geq \frac{\ell_X\abs{{\bigcup_{i}S_i}}}{k+1}-\ell_X+1>\ell_X \left(\frac{\abs{{\bigcup_{i}S_i}}}{k+1}-1\right).
    \end{align*}
\end{proof}

\subsection{Lower bound for Athanasiadis rank of minimal non-very generic intersections}
For an intersection $X=\bigcap_{i=1}^mD_{S_i}\in \L(\B(n,k,\A))$, we define 
\begin{equation*}
    \ell_X=\mathrm{min}_{j\in \bigcup_{i}S_i}\abs{\{i\in [m]\mid j\in S_i\}}.
\end{equation*}

\begin{prop}\label{prop:AthNum}
Let $X=\bigcap_{i=1}^m D_{S_i},\abs{\bigcup_{i=1}^mS_i}=n$ be an intersection in $\L(\B(n,k,\A))$ then the following inequalities hold:
\begin{align*}
a_X&\geq m,\\
a_X&\geq n\ell_X-mk.
\end{align*}
\end{prop}
\begin{proof}
Let $\mathbb{T}=\{S_1,\ldots,S_m\}$ be the canonical presentation of $X$
then we have $a_X=\left(\sum_{S\in \mathbb{T}}\abs{S}\right)-mk$.
By counting the ordered pairs $(i,S)$ such that $i\in S$ in two ways, we obtain
\begin{equation*}
a_X+mk
=\sum_{S\in\mathbb{T}}\abs{S}
=\abs{\{(i,S)\in [n]\times \mathbb{T}\mid i\in S\}}
=\sum_{j=1}^n\abs{\{S\in\mathbb{T}\mid j\in S\}}
\geq n\ell_X.
\end{equation*}
\end{proof}

\begin{lem}\label{prop:Bound_ell}
    Let $X=\bigcap_{i=1}^m D_{S_i}$ be a minimal non-very generic intersection in $\L(\B(n,k,\A))$ then $\ell_X\geq 2$.
\end{lem}
\begin{proof}
    Assume that $\ell_X=1$ , i.e., there is an index ${j_0}\in\bigcup_{i=1}^m S_i$ such that $\{j\in S_i\mid i\in[m], j\in S_i\}=\{S_{i_0}\}$.
Hence we can define the family of indices $\mathbb{T}=\{S_1,\ldots,S_m\}\setminus\{S_{i_0}\}\cup\{S_{i_0}\setminus\{j_0\}\}$ and the intersection $Y=\bigcap_{S\in \mathbb{T}}D_S$.
    By the definition of Athanasiadis rank we get $a_Y=a_X-1$.\\
On the other hand, the orthogonal complement $X^\perp=(\bigcap_{i=1}^m D_{S_i})^\perp\subset\mathbb{S}$ is spanned by $\{\alpha_L\in \mathbb{S}\mid L\subset S_i(i\in [m])\}$ where $L=\{i_1,\ldots,i_{k+1}\}\in\binom{[n]}{k+1}$ and $\alpha_{L}=\det\begin{pmatrix}\mathbf{e}_{{i_1}}&\cdots& \mathbf{e}_{{i_{k+1}}}\\\alpha_{{i_1}}&\cdots&\alpha_{{i_{k+1}}}\end{pmatrix}$.
    When we remove the hyperplane $H_{j_0}$, one coordinate disappears, hence the dimension goes down by one.
    Thus $\rank(Y)=\rank(\bigcap_{S\in \mathbb{T}}D_S)=\rank(X)-1$.
    Therefore $Y$ is still non-very generic and this contradicts the fact that $X$ is a minimal non-very generic intersection.
\end{proof}


\begin{thm}\label{prop:LowerBoundOfAthNum}
    Let $X=\bigcap_{i=1}^m D_{S_i}$ be a minimal non-very intersection in $\L(\B(n,k,\A))$ then $a_X\geq \frac{2\abs{\bigcup_{i=1}^mS_i}}{(k+1)}$.
\end{thm}
\begin{proof}
    By Proposition \ref{prop:AthNum} and Lemma \ref{prop:Bound_ell}, 
    \begin{equation*}
        a_X\geq \ell_X\abs{\bigcup_{i=1}^m S_i}-mk\geq 2\abs{\bigcup_{i=1}^m S_i}-a_Xk.
    \end{equation*}
\end{proof}

\begin{defi}
A minimal non-very generic intersection $X=\bigcap_{i=1}^mD_{S_i}$ in $\L(\B(n,k,\A))$ is \textbf{sparse} if \\ $a_X= \frac{2\abs{\bigcup_{i=1}^mS_i}}{k+1}$.
\end{defi}

\begin{ex}
Given $p\in\mathbb{R}\setminus\{0\}$ and an integer $m\geq 3$, we consider the hyperplane $$H_0 \coloneqq\{(x,y,z)\in\mathbb{R}^3\mid z=px\}$$ and the matrix
    \begin{equation*}
        g\coloneqq\begin{pmatrix}\cos(\pi/m)&-\sin(\pi/m)&0\\\sin(\pi/m)&\cos(\pi/m)&0\\0&0&1\end{pmatrix}
    \end{equation*}
which define a generic arrangement $\A_{2m}:=\{H_i\subset\mathbb{R}^3\mid H_i=g^iH_0\,(i=1,\cdots 2m)\}$.\\
Consider the family of subsets $S_i:=\{i,i+1,i+m,i+m+1\}$ for $i=1,\dots,m$ where each element in $S_i$ is taken modulo $2m$. Then the Athanasiadis rank $a_{\bigcap_{i=1}^m D_{S_i}}$ of the intersection $\bigcap_{i=1}^m D_{S_i}\in\L(\B(2m,3,\A_{2m}))$ is $a_{\bigcap_{i=1}^m D_{S_i}}=2\abs{\bigcup_{i=1}^mS_i}/(k+1)=2\times2m/(3+1)=m$. For the vectors $\alpha_{S_i},i\in[m]$ normal to the hyperplanes $H_i\in\A_{2m}$, we can check that the determinants 
$$\det(\alpha_{i+1}\alpha_{i+m}\alpha_{i+m+1})=\det(\alpha_{i}\alpha_{i+m}\alpha_{i+m+1})=\det(\alpha_{i}\alpha_{i+1}\alpha_{i+m+1})=\det(\alpha_{i}\alpha_{i+1}\alpha_{i+m})$$ are constant for $i\in[m]$. Hence $\alpha_{S_i}, i\in[m]$ are linearly dependent, that is the $\rank \bigcap_{i=1}^m D_{S_i} <m=a_{\bigcap_{i=1}^m D_{S_i}}$ i.e.  $\bigcap_{i=1}^m D_{S_i}$ is non-very generic. In particular the intersection $\bigcap_{i=1}^m D_{S_i}$ is sparse.
\end{ex}
\noindent
As seen in the above example, when $k = 3$, sparse intersections exist for almost all even numbers $\abs{\bigcup_{i=1}^mS_i}$. Another way to construct sparse intersections for any $k$ is to use Settepanella and Yamagata's $K_\mathbb{T}$-configuration in \cite{SeSo}.
But their way of construction implies the condition $n\leq k(k+1)/2$.
Thus sparse intersections for $k\ll n$ cannot be obtained in this way.
One question that arises is: for which $(n, k)$ values does a sparse intersection exist?

\subsection{Upper bound for Athanasiadis rank of minimal non-very generic intersections}
\begin{lem}\label{prop:UpperBoundOfell}
   Let $X$ be a minimal non-very generic intersection in $\L(\B(n,k,\A))$ then $\ell_X\leq k$.
\end{lem}
\begin{proof}
 Without loss of generality, consider $n=\abs{\bigcup_{i=1}^mS_i}$  and assume $\ell_X > k$.
 Fix a hyperplane $H_j$ such that $\abs{\{i\in[m]\mid j\in S_i\}}=\ell_X$ then $X\setminus H_j$ is the  intersection $\left(\bigcap_{i\in[m],j\notin S_i}D_{S_i}\right)\cap\left(\bigcap_{i\in[m],j\in S_i}D_{S_i\setminus j}\right)$. Since $X$ is minimal, the inequality $(n-1)-k>a_{X\setminus H_j}=a_X-\ell_X$ holds, i.e., equivalently
    \begin{equation*}
    	a_X<n+\ell_X-k-1.
    \end{equation*}
    The above condition and the inequalities of Proposition \ref{prop:AthNum} yield
    \begin{align*}
    	m&<n+\ell_X-k-1,\\
    	n\ell_X-km&<n+\ell_X-k-1.
    \end{align*}
    Furthermore, the above two inequalities are equivalent to inequalities  $$(n\ell_X-n-\ell_X+k+1)/k<m<n+\ell_X-k-1$$ which imply $$(n\ell_X-n-\ell_X+k+1)-k(n+\ell_X-k-1)=(n-k-1)(\ell_X-k-1)<0 \quad .$$
 Since $\ell_X\geq k+1$, it follows that $n-k-1<0$, i.e., $n \leq k$ and we get an absurd since all generic arrangements consisting of $k$ hyperplanes in a $k$-dimensional space are very generic. Hence $\ell_X\leq k$.
\end{proof}

\begin{thm}\label{prop:UpperBoundOfAthNumb}
    Let $X$ be a minimal non-very generic intersection in $\L(\B(n,k,\A))$ then $a_X\leq \rank (X)+\ell_X-1$.
    In particular, $\rank(X)<a_X\leq \rank(X)+k-1$.
\end{thm}
\begin{proof}
    Consider a hyperplane $H_j$ such that $\abs{\{i\in[m]\mid j\in S_i\}}=2$ then $X\setminus H_j$ is the intersection $\bigcap_{i\in[m],j\notin S_i}D_{S_i}$ and we have that $\rank (X)\geq \rank(X\setminus H_j)+1=a_{X\setminus H_j}+1=a_X-\ell_X+1$.
\end{proof}

\noindent
In the following, we will show that the inequalities in Theorem \ref{prop:UpperBoundOfAthNumb} are sharp.
\begin{defi}
A minimal non-very generic  intersection $X=\bigcap_{i=1}^mD_{S_i}$ in $\L(\B(n,k,\A))$  is \textbf{dense} if\\ $a_X=\rank X+k-1$.
\end{defi}

\begin{thm}
For any $k\geq 2$, there exists a non-very generic arrangement $\A$ consisting of a sufficiently large number of hyperplanes in $\mathbb{R}^k$ (or $\mathbb{C}^k$) such that $\L(\B(n,k,\A))$ contains a dense intersection.
\end{thm}
\begin{proof}
Let $k$ be a positive integer and denote by $[\pm k]$ the set $\{-k,\dots,-1,1,\dots,k\}\subset \mathbb{Z}$. Consider an arrangement $\mathcal{A}$ of $n=2k+2$ hyperplanes indexed by $[\pm (k+1)]$ and set $\nu(\mathbb{T})=\sum_{S\in\mathbb{T}}(\abs{S}-k)$, $\mathbb{T}\subset 2^{[\pm k]}$. Note that, if $\mathbb{T}$ is the canonical presentation of an intersection $X\in\mathcal{L}(\mathcal{B}(n,k,\mathcal{A}))$ then $\nu(\mathbb{T})=a_X$.
When $k=2$, Crapo's example satisfies the equation $a_X=\rank X+k-1$ and hence Crapo's arrangement is our required non-very generic arrangement containing a dense intersection. Now we check the case when $k>2$.\\
Define three subsets of $\binom{[\pm(k+1)]}{k+1}$ as follows:
\begin{align*}
\mathbb{U}^+_k=&\{K_{k,i}=[k+1]\setminus\{i\}\cup\{-i\}\mid i\in[k]\},\\
\mathbb{U}^-_k=&\{K_{k,-i}=[\pm(k+1)]\setminus K_i\mid i\in[k]\},\\
\mathbb{U}_k=&\mathbb{U}_k^+\sqcup\mathbb{U}_k^-.
\end{align*}
For instance $\mathbb{U}_2=\{K_{2,1}=\{-1,2,3\},K_{2,2}=\{1,-2,3\},K_{2,-1}=\{1,-2,-3\},K_{2,-2}=\{-1,2,-3\}\}$ corresponds to Crapo's case.\\
By construction $\nu(\mathbb{U}_k)=2k$, hence the goal is to show the existence of an intersection $X$ whose canonical presentation is $\mathbb{U}_k$, i.e. $a_X=2k$ and whose rank is $k+1$, i.e. $a_X=\rank X+k-1$ that is X is dense.\\
For each $i\in[\pm(k+1)]$ the equality $\abs{\{L\in\mathbb{U}_k\mid i\in L\}}=\abs{\{L\in\mathbb{U}_k\mid i\notin L\}}=k$ holds and, if $\mathrm{sgn}:\mathbb{Z}\to \{-1,0,1\}$ denotes the sign function, for each $i,j\in [\pm k]$ we have
\begin{equation*}
\abs{K_{k,i}\cap K_{k,j}}=\left\{\begin{aligned}
&0&&i+j=0;\\
&2&&i+j\neq0\mbox{ and }\mathrm{sgn}(i)=-\mathrm{sgn}(j);\\
&k-1&&\mbox{otherwise} .\\
\end{aligned}\right.
\end{equation*}
Hence the following statement holds: \\
\\
\textbf{Fact 1.} \textit{Let $\mathbb{T}$ be a subset of $\mathbb{U}_k$ of cardinality $k+1$.
Then $\abs{\bigcup_{L\in \mathbb{T}}L}=2k+2$.}\\ 
\\
\noindent
The symmetric group $\mathfrak{S}_k$ of degree $k$ acts naturally on $\{1,\dots, k\}$ and if $\sigma\in\mathfrak{S}_k$ acts on the set $\{i_1,\dots,i_l\}$ as $\sigma(\{i_1,\dots,i_l\})=\{\sigma(i_1),\dots,\sigma(i_l)\}$ then we set $\sigma(-i_j)=-\sigma(i_j), j=1,\ldots,l$. In this way we define an action of $\mathfrak{S}_k$ on $2^{[\pm k]}$ which satisfies
\begin{equation*}
\sigma(K_{k,i})=[k+1]\setminus\{\sigma(i)\}\cup\{-\sigma(i)\}=K_{k,\sigma(i)} \quad .
\end{equation*}
In order to complete the proof, we need to prove the following:\\
\\
\textbf{Fact 2.} For any fixed $k\geq 3$ and $\mathbb{T}\subset\mathbb{U}_k$,
if $2\leq \abs{\mathbb{T}}\leq n-k-1$ then $\abs{\mathbb{T}}+k<\abs{\bigcup_{L\in\mathbb{T}}L}$ and $\mathbb{T}$ satisfies the BBA-condition.\\
\textit{[Proof of Fact 2.]} First notice that since $n=2k+2$ then $n-k-1=k+1$.
We use induction on $k$. When $k=3$, then $\abs{\mathbb{T}}$ is either $2,3$ or $4$. In the first case $\mathbb{T}=\{L_1,L_2\}$, then
\begin{equation*}
	\abs{L_1\cup L_2}=\abs{L_1}+\abs{L_2}-\abs{L_1\cap L_2}\geq2(k+1)-(k-1)=k+3>2+k=\abs{\mathbb{T}}+k.
\end{equation*}
In the second case, $\mathbb{T}=\{L_1,L_2,L_3\}$ and we have either $\abs{L_1\cap L_2\cap L_3}>0$, then
\begin{align*}
	\abs{L_1\cup L_2\cup L_3}=&\abs{L_1}+\abs{L_2}+\abs{L_3}-\abs{L_1\cap L_2}-\abs{L_2\cap L_3}-\abs{L_1\cap L_3}+\abs{L_1\cap L_2\cap L_3}\\
	\geq&3(k+1)-3(k-1)+1=7>k+3=\abs{\mathbb{T}}+k
\end{align*}
or $\abs{L_1\cap L_2\cap L_3}=0$, then there exist $i\in[k]$ and $j\in[\pm k]$ such that $\mathbb{T}=\{K_{k,i},K_{k,-i},K_{k,j}\}$.
Thus we get
\begin{align*}
	&\abs{K_{k,i}\cup K_{k,-i}\cup K_{k,j}}\\
	=&\abs{K_{k,i}}+\abs{K_{k,-i}}+\abs{K_{k,j}}-\abs{K_{k,i}\cap K_{k,-i}}-\abs{K_{k,-i}\cap K_{k,j}}-\abs{K_{k,i}\cap K_{k,j}}+\abs{K_{k,i}\cap K_{k,-i}\cap K_{k,j}}\\
	\geq&3(k+1)-0-2(k-1)=k+5>k+3=\abs{\mathbb{T}}+k.
\end{align*}
In the last case $\abs{\mathbb{T}}=3+1$ and by Fact 1 we get $\abs{\bigcup_{L\in\mathbb{T}}L}=2k+2>\abs{\mathbb{T}}+k$ .\\
Let's assume that Facts 2 holds for $k> 3$ and fix $\mathbb{T}\subset \mathbb{U}_{k}$ with $2\leq \abs{\mathbb{T}}\leq k+1$.
We define $\mathbb{T}^{(-1)}\coloneqq\{L\setminus \{k+1,-(k+1)\}\mid L\in\mathbb{T}\setminus\{K_{k,k},K_{k,-k}\}\}$.
Notice that $\mathbb{U}_{k-1}=\mathbb{U}_k^{(-1)}$.\\
Suppose there exists an index $i_0\in[k]$ such that both $K_{k,i_0}$ and $K_{k,-i_0}$ are not in $\mathbb{T}$.
Since $\mathfrak{S}_k$ acts on the subset $\mathbb{U}_k$, we can assume $i_0=k$ without loss of generality.
If both $K_{k,k}$ and $K_{k,-k}$ are not in $\mathbb{T}$ then 
\begin{equation*}
\abs{\mathbb{T}}+k=\abs{\mathbb{T}^{(-1)}}+(k-1)+1<\abs{\bigcup_{L\in \mathbb{T}^{(-1)}}L}+1\leq \abs{\bigcup_{L\in \mathbb{T}}L}.
\end{equation*}
Conversely, if one of $K_{k,i}$ or $K_{k,-i}$ is contained in $\mathbb{T}$ for all $i\in[k]$ then $\abs{\mathbb{T}}=k$ or $k+1$. When $\abs{\mathbb{T}}=k+1$, using Fact 1 it is easy to check that $\mathbb{T}$ satisfies the BBA-condition.
Now we just need to check the case when $\abs{\mathbb{T}}=k$.
In this case
\begin{equation*}
\abs{\bigcup_{L\in \mathbb{T}}L}=\left\{\begin{aligned}
&2k+2&&\mbox{ if }\abs{\mathbb{U}^+_k\cap\mathbb{T}}\geq 2\mbox{ and }\abs{\mathbb{U}^-_k\cap\mathbb{T}}\geq 2,\\
&2k+1&&\mbox{ otherwise .}
\end{aligned}\right.
\end{equation*}
Hence we have $\abs{\mathbb{T}}+k=2k<\abs{\bigcup_{L\in\mathbb{T}}L}$. Therefore we get $\nu(\mathbb{T}')=\abs{\mathbb{T}'}<\abs{\bigcup_{S\in\mathbb{T}'}S}-k$ for each $\mathbb{T}'\subset\mathbb{T}$ and $\mathbb{T}$ satisfies the BBA-condition. This concludes the proof of Fact 2 and we can go back to the proof of the theorem.\\

\noindent
Let's construct a minimal non-very generic intersection $X$ whose canonical presentation is $\mathbb{U}_k$.\\
Let $v_{-k},\ldots,v_{-1},\allowbreak v_1,\ldots,v_k$ be non-zero vectors in $\mathbb{R}^k$ $(\mathbb{C}^k)$ and for each $i\in [\pm(k+1)]$,  define the affine hyperplane $H'_i$ as the closure of $\{v_j\in\mathbb{R}^k\mid i\in K_{k,j}\}$, that is
\begin{equation*}
	H'_i=\left\{\sum_{j:i\in K_{k,j}}v_jx_j\in\mathbb{R}^k\;\left|\; \sum_{j:i\in K_{k,j}}x_j=1, x_j\in \mathbb{R}\right.\right\}.
\end{equation*}
If we denote by $\alpha_i$ a vector normal to $H_i'$ and define the linear hyperplane arrangement $\mathcal{A}=\{H_i\mid H_i\perp \alpha_i, i\in[\pm(k+1)]\ \}$ in $\mathbb{R}^k(\mathbb{C}^k)$, then there exists a Zariski closed set in the set of nonzero vector lists $(v_i)_{i\in[\pm k]}$ in $\mathbb{R}^k(\mathbb{C}^k)$ such that for each element in this Zariski closed set makes $\mathcal{A}$ is not generic.
We take $(v_i)_{i\in[\pm k]}$ from the complement of such Zariski closed set.
By construction, $\mathcal{B}(2k+2,k,\mathcal{A})$ has an intersection $X=\bigcap_{L\in\mathbb{U}_k}D_L$ whose canonical presentation is $\mathbb{U}_k$.\\
Finally we check the existence of $(v_i)_{i\in[\pm k]}$ such that the non-very generic intersection $X$ is minimal.
If $X$ is not minimal then there exists $\mathbb{T}\subset\mathbb{U}_k$ such that $2\leq\abs{\mathbb{T}}\leq n-k-1=k+1$ and $\rank\bigcap_{L\in\mathbb{T}} D_L<\abs{\mathbb{T}}$.
However, since such a set of nonzero vector lists $(v_i)_{i\in[\pm k]}$ in $\mathbb{R}^k(\mathbb{C}^k)$ becomes Zariski closed again, we can choose $(v_i)_{i\in[\pm k]}$ such that $X$ is minimal.
Then $X$ is a minimal non-very generic intersection of rank $n-k-1=k+1$ and Athanasiadis rank $a_X=2k=\abs{\mathbb{U}_k}$. It follows that $a_X=\rank(X)+k-1$ and $X$ is a dense intersection.
\end{proof}
\noindent
The following proposition holds.

\begin{prop}
   Let $X$ be a minimal non-very generic intersection in $\L(\B(n,k,\A)).$ If $X$ is dense then $\rank X\geq k+1$.
\end{prop}
\begin{proof}
Without loss of generality, let's assume that $n=\abs{\bigcup_{L\in\mathbb{T}} S}$ and denote by $\mathbb{T}=\{S_1,\dots,S_m\}$ the canonical presentation of $X$. Recall that $a_X=\rank(X)+k-1$ and that the $\mathrm{min}_{j\in \bigcup_{S\in\mathbb{T}}S}\abs{\{S\in\mathbb{T}\mid j\in S\}}=k$.
By Proposition \ref{prop:Bound_ell} and Theorem \ref{prop:UpperBoundOfAthNumb} the following inequalities hold:
\begin{equation*}
\rank X+(m+1)k-1=a_X+mk
\geq n\ell_X\geq nk.
\end{equation*}
Hence, we have that $\rank X\geq (n-m-1)k+1$.
Since $X$ is minimal, $\mathbb{T}\setminus j:= \{S\setminus j\mid S\in \mathbb{T}\}\setminus \binom{[n]}{k}$ has to satisfy BBA-condition for any $j\in \bigcup_{S\in \mathbb{T}} S$, that is
\begin{align*}
n-k-1=\abs{\bigcup_{S\in\mathbb{T}}S\setminus\{j\}}-k=\abs{\bigcup_{S\in \mathbb{T}\setminus j}S}-k
> a_{\left(\bigcap_{L\in \mathbb{T}\setminus j}D_S \right)}= a_X-\abs{\{S\in \mathbb{T}\mid j\in S\}}.
\end{align*}
In particular, since $\mathrm{min}_{j\in \bigcup_{S\in\mathbb{T}}S}\abs{\{S\in\mathbb{T}\mid j\in S\}}=k$ and $a_X \geq m$, it follows that 
\begin{equation*}
n-k-1> a_X-k\geq m-k.
\end{equation*}
Adding $k$ to both sides, we get $n-1>m$ and hence $\rank X\geq (n-m-1)k+1\geq k+1$.
\end{proof}
\noindent
\noindent


\section{Description of rank $2$ intersections in the discrminantal arrangements}\label{sec:rank2}
In this section, we correct the result on rank $2$ intersections in discriminantal arrangements provided by Libgober and the third author in \cite{LS} which stated that the multiplicity of the rank $2$ intersections in discriminantal arrangements $\B(n,k,\A)$ is $2,3$ or $k+2$ (see \cite{LS}). We show that we can construct intersections of rank $2$ with multiplicity up to $\frac{n}{n-k-1}$.

\begin{thm}\label{prop:Rank2AthNum}
Let $X=\bigcap_{i=1}^mD_{L_i}$ be an intersection in $\L(\B(n,k,\A)), n=\mid \bigcup_{i=1}^m L_i \mid$.
If the $\rank(X)=2$ then $X$ satisfies $a_X\leq \frac{n}{n-k-1}$. Moreover if $\frac{n}{n-k-1}$ is an integer then there exists an arrangement consisting of $n$ hyperplanes in $\mathbb{R}^k$ with non-very generic intersection $X$ satisfying $a_X= \frac{n}{n-k-1}$.
\end{thm}
\noindent
\begin{proof}If $X$ is very generic then the statement holds.
Let $X=\bigcap_{i=1}^mD_{L_i}$ be a rank $2$ non-very generic intersection, $\{L_1,\dots,L_m\}$ canonical presentation, then
\begin{equation}\label{eq:RankAndCard}
    L_{i_1}\cup L_{i_2}=\bigcup_{i=1}^m L_i, \quad 1\leq i_1< i_2\leq m.
\end{equation} 
Indeed if $j \in \bigcup_{i=1}^m L_i\setminus(L_{i_1}\cup L_{i_2}) \neq \emptyset$ then there exists $i_o\in[m]$ such that $j \in L_{i_0}$ and $D_{L_{i_0}}\supset X=D_{L_{i_1}}\cap D_{L_{i_2}}\supset D_{L_{i_1}\cup L_{i_2}}$ which is an absurd. From equation \eqref{eq:RankAndCard} it follows that $L_i\cup L_j=[n]$, for each $i,j\in[m]$ 
or, equivalently, $[n]\setminus L_i\subset L_j$.
Thus the inequality 
\begin{equation*}
    n=\abs{[n]}=\sum_{i=1}^m \abs{[n]\setminus L_i}+\abs{\bigcap_{i=1}^m L_i}\geq \sum_{i=1}^m \abs{[n]\setminus L_i}=m(n-k-1) 
\end{equation*}
holds and hence $a_X\leq \frac{n}{n-k-1}$. Finally, since $\frac{n}{n-k-1}\leq k+1$ holds for $n-2>k>1$, Theorem \ref{prop:Rank2AthNum} is obtained as a corollary of Theorem \ref{prop:UpperBoundOfAthNumb} for rank $2$ intersections.
\end{proof}
\noindent
Notice that every rank $2$ non-very generic intersections in $\L(\B(n,k,\A))$ are simple,
i.e. the canonical presentation of $X=\bigcap_{i=1}^m D_{L_i}$ satisfies $\abs{L_i}=k+1$ for $i=1,\dots,m$ and $a_X=m$.\\
The following arrangement of $8$ hyperplanes in $\mathbb{R}^5$ is the first example containing a rank 2 intersection having multiplicity different from the ones mentioned in the result by Libgober and the third author.

\begin{ex}
    Let $\A=\{H_1,\ldots,H_8\}$ be the central arrangement of hyperplanes $H_i$'s orthogonal to $\alpha_i$'s defined by:
    \begin{equation*}
        \left(\alpha_1\alpha_2\alpha_3\alpha_4\alpha_5\alpha_6\alpha_7\alpha_8\right)=
        \begin{pmatrix}
        1&0&0&0&0&1&2&3\\
        0&1&0&0&0&1&-2&-1\\
        0&0&1&0&0&1&3&2\\
        0&0&0&1&0&1&-1&-2\\
        0&0&0&0&1&1&1&1\\
        \end{pmatrix}.
    \end{equation*}
The arrangement $\A$ is generic and $X=D_{123456}\cap D_{123478}\cap D_{125678}\cap D_{345678} \in \L(\B(8,5,\A))$ has canonical presentation given by \{\{1,2,3,4,5,6\},\{1,2,3,4,7,8\},\{1,2,5,6,7,8\},\{3,4,5,6,7,8\}\}, that is $a_X=4$. Moreover it is an easy computation to verify that the rank of the matrix 
    \begin{equation*}
        \begin{pmatrix}
        \alpha_{123456}\\ \alpha_{123478}\\ \alpha_{125678}\\ \alpha_{345678}
        \end{pmatrix}=
        \begin{pmatrix}
        1&1&1&1&1&-1&0&0\\
        1&1&-1&-1&0&0&1&-1\\
        -8&-8&0&0&-4&4&-4&4\\
        0&0&8&8&4&-4&-4&4\\
        \end{pmatrix}.
    \end{equation*}
defined by the vectors orthogonal to $D_{123456}, D_{123478}, D_{125678}, D_{345678}$ is equal to $2$.  That is $X$ is an intersection of $\rank(X)=2$ and multiplicity $4$ in $\B(8,5,\A)$.
\end{ex}
\noindent
It is possible to provide a family of such examples generalizing the Falk's adjoint construction (see \cite{Falk}).
\begin{defi}
    Let $\{\alpha_1,\ldots,\alpha_n\}$ be a set of vectors in $\mathbb{R}^k$.
    The \textbf{Gale diagram} of $\{\alpha_1,\ldots,\alpha_n\}$ is a set $\{\beta_1,\ldots,\beta_n\}$ of vectors in $\mathbb{R}^{n-k}$ such that the row space of $(\beta_1\ldots\beta_n)$ is the orthogonal complement of the row space of $(\alpha_1\ldots\alpha_n)$.
\end{defi}
\noindent
Remark that, if $\{\alpha_1,\ldots\alpha_n\}$ is in general position then its Gale diagram is also in general position. Also, it is known that the Gale diagram construction is a dual operation.
For more properties of Gale diagrams, refer to \cite{Zie}.
\begin{thm}[see \cite{Falk}]
    Let $\A$ be a generic arrangement with normals $\{\alpha_1,\ldots\alpha_n\}$ in $\mathbb{R}^k$.
    The essential part of $\B(n,k,\A)$ is linearly equivalent to the arrangement consisting of all hyperplanes spanned by $n-k-1$ elements in the Gale diagram of $\{\alpha_1,\ldots\alpha_n\}$.
\end{thm}

\noindent
By the above theorem, it follows that if we can construct an arrangement generated by $n$ points $\beta_1,\ldots,\beta_n$ in an $(n-k)$-dimensional space which intersection lattice contains an element $X$ of rank $2$ and multiplicity $\frac{n}{n-k-1}$, then the inequality in Theorem \ref{prop:Rank2AthNum} is optimal. Notice that $n$ and $k$ have to satisfy that $\frac{n}{n-k-1}$ is an integer. Let's start this construction assuming $n-k-1=2$.
\begin{ex}\label{ex:n-k=3rank2}\textbf{Case $n-k-1=2$.} Let $n\geq 6$ be an even number, set $k=n-3$ and consider the points
    \begin{equation*}
        (\beta_1\beta_2\ldots\beta_n)=\begin{pmatrix}
            \cos\frac{2\pi}{n}&\cos\frac{2\cdot 2\pi}{n}&\dots&\cos\frac{j\cdot 2\pi}{n}&\dots&\cos\frac{n\cdot 2\pi}{n}\\
            \sin\frac{2\pi}{n}&\sin\frac{2\cdot 2\pi}{n}&\dots&\sin\frac{j\cdot 2\pi}{n}&\dots&\sin\frac{n\cdot 2\pi}{n}\\
            1&1&\dots&1&\dots&1
        \end{pmatrix}.
    \end{equation*}
    It is easy an easy computation to verify that $\beta_1,\ldots,\beta_n$ are in general position, i.e. $\det(\beta_{i_1}\beta_{i_2}\beta_{i_3})\neq 0$ for all $1\leq i_1<i_2<i_3\leq n$. The $1$-dimensional subspace $X=\langle(0,0,1)^\mathsf{T}\rangle$ is included in all subspaces $\langle\beta_i,\beta_{i+\frac{n}{2}}\rangle, i=1,\dots,\frac{n}{2}$.
    Hence the arrangement generated by $\beta_1,\ldots,\beta_n$ has a rank $2$ intersection $X$ of  multiplicity $\frac{n}{2}=\frac{n}{n-k-1}$. The Gale diagram $\{\alpha_1,\ldots,\alpha_n\}$ of $\{\beta_1,\ldots,\beta_n\}$ defines the arrangement $\{H_1,\dots,H_n\}$ with an intersection of multiplicity $\frac{n}{2}=\frac{n}{n-k-1}$ in the rank $2$ intersection set of its discriminantal arrangement.
\end{ex}
\noindent
The following example provides the general construction.
\begin{ex}\label{ex:generalrank2} Choose $n$ and $k$ such that  $n-k-1$ divides $n$. Let $X$ be a linear subspace of dimension $n-k-2$ in $\mathbb{R}^{n-k}$ and $D_1,\dots,D_{\frac{n}{n-k-1}}$ be hyperplanes in $\mathbb{R}^{n-k}$ such that $D_i\supset X$ for each $i=1,\dots,\frac{n}{n-k-1}$.
For each $i=1,\dots,\frac{n}{n-k-1}$, fix $n-k-1$ vectors $\beta_{i,1},\ldots,\beta_{i,n-k-1} \in D_i$ in the complement of $X$. Since the set of generic arrangements in the space of arrangements is a Zariski open set it is possible to choose $\beta_{i,j}$'s such that $\{\beta_{i,j}\mid 1\leq i\leq \frac{n}{n-k-1},1\leq j\leq n-k-1\}$ are in general position. The hyperplanes $D_i=\langle\beta_{i,1},\ldots,\beta_{i,n-k-1}\rangle, i=1,\dots,\frac{n}{n-k-1}$ intersect in a rank $2$ intersection $X$ of multiplicity $\frac{n}{n-k-1}$ providing the arrangement we want to construct.
\end{ex}
\noindent
\begin{rem}
Only the part 3 in Theorem 3.9 in \cite{LS} requires to be modified while 
the parts 1,2 and 4 hold. In particular we retrieve the part 2 as a consequence of Equation \eqref{eq:RankAndCard} which states that if three hyperplanes $D_{L_1}, D_{L_2}, D_{L_3}$ intersect in rank $2$, then $L_1 \subset L_2 \cup L_3$, $L_2 \subset L_1 \cup L_3$, and $L_3 \subset L_1 \cup L_2$.
Therefore, by setting $L'_i = L_i \setminus \bigcap_{j=1}^3 L_j$, we obtain the good-$3$-partition $\{L'_1 \cap L'_2, L'_1 \cap L'_3, L'_2 \cap L'_3\}$, which corresponds to the multiplicity $3$ intersection of rank $2$ in the discriminantal arrangement determined by the dependent restriction $\mathcal{A}^{X}$, $X=\bigcap_{i \in L_1 \cap L_2 \cap L_3} H_i$.\\
Notice that in terms of Gale duality, the restriction $\mathcal{A}^{X}$, $X=\bigcap_{i \in L_1 \cap L_2 \cap L_3} H_i$, is equivalent to the deletion of $\{\beta_i \mid i \in \bigcap_{j=1}^3 L_j\}$ from $\{\beta_i \mid i=1, \ldots, n\}$.
\end{rem}

\end{document}